\documentclass[11pt,leqno]{amsart}
\usepackage{amssymb,amsmath,amsthm,latexsym}
\usepackage{graphicx}
\usepackage{epsfig}
\newtheorem{theorem}{Theorem}[section]
\newtheorem{lemma}[theorem]{Lemma}

\theoremstyle{remark}
\newtheorem{definition}[theorem]{Definition}

\newtheorem{example}[theorem]{Example}

\newcommand\eL{{\mathcal L}}
\newcommand\M{{\mathcal M}}
\newcommand\R{{\mathcal R}}

\newcommand\xt{{(x_0,x_1,\ldots,x_n,x_{n+1})}}

\begin{document}

\title{Layers of knot region colorings and higher differentials}
\date{January 28, 2019}
\author{Maciej Niebrzydowski}
\address[Maciej Niebrzydowski]{Institute of Mathematics\\ 
Faculty of Mathematics, Physics and Informatics\\
University of Gda{\'n}sk, 80-308 Gda{\'n}sk, Poland}
\email{mniebrz@gmail.com}


\thispagestyle{empty}

\begin{abstract}
We inductively define layers of colorings of knot and knotted surface diagrams using ternary quasigroups. Homological invariants from such systems of colorings use shorter differentials and of higher degree than the standard homology differentials, and give access to typically more complex homology groups.
\end{abstract}

\maketitle

\section{Introduction and preliminary definitions}

Shadows in quandle colorings of knot and knotted-surface diagrams \cite{CKS01,FRS95,Joy82,Kam02,KLT15,Mat82,RS00} allow one to use, respectively, the third and the fourth quandle (co)homology groups, instead of the second and the third, in defining invariants.

For region colorings of knot and knotted surface diagrams with elements of ternary quasigroups, we inductively define multiple layers of colorings that come in two types (left and right colorings). We then use such systems of colorings to define (co)homological invariants based on (co)homology groups of ternary quasigroups of arbitrarily high degree, but with a shorter differential than the standard one.

Let us begin with some preliminary definitions.

\begin{definition}\label{quasigroup}
A {\it ternary quasigroup} is a set $X$ equipped with a ternary operation 
$[\ ]\colon X^3\to X$ such that for a quadruple $(x_1,x_2,x_3,x_4)$
of elements of $X$ satisfying $[x_1x_2x_3]=x_4$, specification of any three elements of the quadruple determines the remaining one uniquely. This leads to three additional ternary operations $\eL$, $\M$, $\R\colon X^3\to X$, defined via
\[x_4x_2x_3\eL=x_1,\ x_1x_4x_3\M=x_2,\ \textrm{and}\ x_1x_2x_4\R=x_3.\]
We call them the left, middle, and right division, respectively.
\end{definition}

A finite ternary quasigroup $(X,[\ ])$, with elements numbered $1,\ldots,n$, can be described by a Latin cube, i.e., an $n\times n\times n$ array in which every $i\in\{1,\ldots,n\}$ appears exactly once in every horizontal row, every vertical row, and in every column. Any Latin cube defines a ternary quasigroup. 
See \cite{Bel,BelSan,Sm08} for more details on $n$-ary quasigroups.

\begin{figure}
\begin{center}
\includegraphics[height=5 cm]{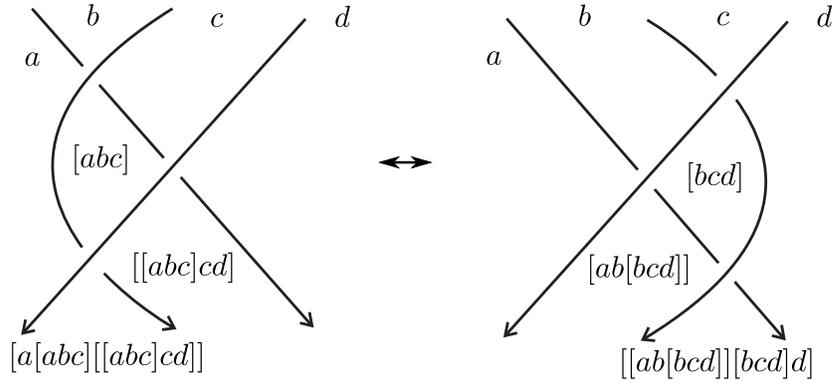}
\caption{Comparing the bottom colors leads to the nesting conditions.}\label{rad3br}
\end{center}
\end{figure}

\begin{definition}(\cite{Nie14}) \label{KTQ}
A knot-theoretic ternary quasigroup (abbreviated to KTQ) is a ternary quasigroup satisfying the left and right nesting conditions derived from the third Reidemeister move (see Fig. \ref{rad3br}): 
\begin{equation}
\forall_{a,b,c,d\in X} \quad [ab[bcd]]=[a[abc][[abc]cd]],  \tag{LN}
\end{equation}
\begin{equation}
\forall_{a,b,c,d\in X} \quad [[abc]cd]=[[ab[bcd]][bcd]d]. \tag{RN} 
\end{equation}
\end{definition}

\begin{example}(\cite{Nie14})
Let $(X,*)$ be a group. Then $(X,[\ ])$, where $[xyz]=x*y^{-1}*z$ is a KTQ. This operation is used in the Dehn presentation of the knot group.
\end{example}

\begin{example}
Let $(X,*)$ be a group. Then $(X,[\ ])$ with $[xyz]=y*z^{-1}*\alpha*x$, where $\alpha$ is an arbitrary fixed element of $X$, is a KTQ. 
\end{example}

In this paper, we use the term knot (resp. knotted surface) for both knots (resp. surface-knots) and links (resp. surface-links).

The conditions LN and RN, together with ternary quasigroups motivated by the Dehn presentation of the knot group (which does not require orientation), were used to define knot invariants in \cite{Nie14}.
In \cite{NeNe17} the authors used compositions of two binary quasigroup operations of the form $x*(y\cdot z)$, satisfying the conditions LN and RN, to define invariants for oriented knots. Ternary quasigroups with operation $[\ ]$ decomposable in this way, together with the ones possessing the opposite decomposition $[xyz]=(x*y)\cdot z$, form the family known in the literature as {\it reducible ternary quasigroups}. We also mention the paper \cite{Dev09}, in
which the author constructed combinatorial invariants of knots based on colorings of regions of a knot diagram by elements of some finite ring $R$, with coloring requirements involving the equation $pa+b-c-pd=0$, for $a$, $b$, $c$, $d\in R$ and an invertible element $p\in R$.
For general colorings with ternary quasigroups and their homology see \cite{Nie17a}; we will briefly recall the relevant definitions in this introduction.

\begin{definition}
Let $(X,[\ ])$ and $(Y,<\ >)$ be two ternary quasigroups. We say that a function
$f\colon X\to Y$ is a ternary quasigroup homomorphism if 
\[
f([abc])=<f(a)f(b)f(c)>,
\]
for any $a$, $b$, $c\in X$.
\end{definition}

A computer search (using GAP \cite{GAP4}) for Latin cubes satisfying the conditions LN and RN indicates that, up to isomorphism, there are 2 KTQs of size two, 7 of size three, 37 with four elements, 23 with five elements,
and over 190 with six elements (in this case the list could be far from being complete).

We will now recall the definition of KTQ colorings of diagrams of knots and knotted surfaces. For the terminology related to (broken) diagrams of knotted surfaces in $\mathbb{R}^3$ see, for example, \cite{KnSurf}.

\begin{definition}
Let $D$ denote a knot diagram in the interior of a compact surface $F$, or on a plane, or a knotted surface diagram in $\mathbb{R}^3$.
{\it Regions} of $D$ are the connected components of the complement of the universe of $D$ in, respectively, $F$, the plane, or $\mathbb{R}^3$. Their set will be denoted by $Reg(D)$.
\end{definition}

\begin{definition}
Let $D$ be an oriented knot diagram. In the rest of the paper
$e$ will denote an edge of $D$ (as in the underlying graph of $D$), and $cr$ will denote a crossing. If $D$ is an oriented knotted surface diagram, then $db$ will denote a double point edge, and $tr$ is a notation for a triple point.
For a given $e$, $cr$, $db$, or $tr$, its {\it source region} is the region in its neighborhood such that all the co-orientation arrows point away from it, and the {\it target region} is the region with all co-orientation arrows pointing into it.
\end{definition}

\begin{figure}
\begin{center}
\includegraphics[height=3.7 cm]{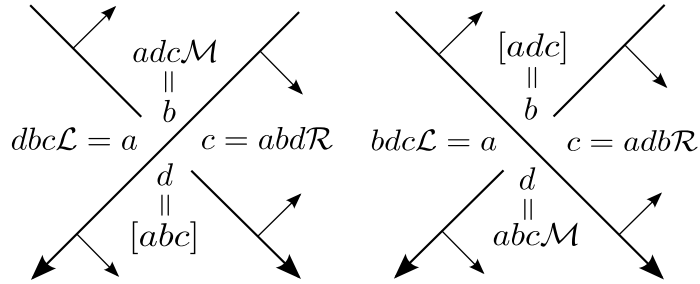}
\caption{Coloring classical crossings with a KTQ.}\label{quasicol}
\end{center}
\end{figure}

\begin{figure}
\begin{center}
\includegraphics[height=4 cm]{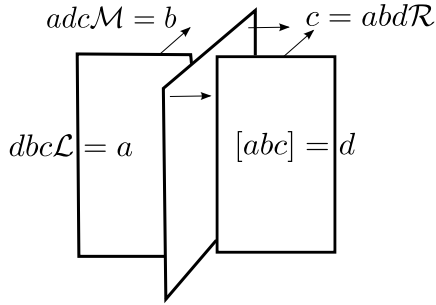}
\caption{The coloring rule around an edge of double points in a surface diagram.}\label{surfcr}
\end{center}
\end{figure}

The following definition will be helpful in defining KTQ colorings, and in homological considerations.

\begin{definition}\label{colpath}
Let $(X,[\ ])$ be a KTQ, $D$ be an oriented knot or knotted surface diagram,
and let $\mathcal{C}\colon Reg(D)\to X$ be a function assigning elements of $X$ to the regions of $D$. Then, for a given $e$, $cr$, $db$, or $tr$ of $D$:
\begin{enumerate}
\item[--] a colored path of $e$, denoted by $cp(e,\mathcal{C})$, is the pair $(\mathcal{C}(r_0),\mathcal{C}(r_1))$, where $r_0$ is the source region of $e$, and $r_1$ is its target region;
\item[--] a colored path of $cr$, denoted by $cp(cr,\mathcal{C})$, is the triple $(\mathcal{C}(r_0),\mathcal{C}(r_1),\mathcal{C}(r_2))$, where $r_0$ is the source region of $cr$, $r_0$ and $r_1$ are separated by an under-arc of $cr$, and $r_2$ is the target region of $cr$;
\item[--] a colored path of $db$, denoted by $cp(db,\mathcal{C})$, is the triple $(\mathcal{C}(r_0),\mathcal{C}(r_1),\mathcal{C}(r_2))$, where $r_0$ is the source region of $db$, $r_0$ and $r_1$ are separated by an under-sheet of $db$, and $r_2$ is the target region of $db$;
\item[--] a colored path of $tr$, denoted by $cp(tr,\mathcal{C})$, is 
$(\mathcal{C}(r_0),\mathcal{C}(r_1),\mathcal{C}(r_2),\mathcal{C}(r_3))$, where $r_0$ is the source region of $tr$, $r_0$ and $r_1$ are separated by the bottom sheet, $r_1$ and $r_2$ by the middle sheet, and $r_3$ is the target region for $tr$.
\end{enumerate}
\end{definition}

\begin{definition}\label{manycols}
Let $D$ denote an oriented knot diagram on a plane or on an oriented compact surface $F$, or a knotted surface diagram in $\mathbb{R}^3$. Let $(X,[\ ],\eL,\M,\R)$ be a KTQ. For any crossing $cr$ of $D$ (resp. double point edge $db$ of $D$) and its regions $r_0$, $r_1$, $r_2$ as in Def. \ref{colpath}, let $r_3$ be the remaining fourth region in the neighborhood of $cr$ (resp. $db$).
For a map $\mathcal{C}$ to be a {\it KTQ coloring} of $D$, it is required that
$\mathcal{C}(r_3)=[\mathcal{C}(r_0)\mathcal{C}(r_1)\mathcal{C}(r_2)]$. To put it succintly: the value of the operation 
$[\ ]$ on the colored path of $cr$ (resp. $db$) is assigned to the fourth region of $cr$ (resp. $db$). See Fig. \ref{quasicol} and Fig. \ref{surfcr}.
\end{definition}

\begin{theorem}(\cite{Nie17a})
The number of KTQ colorings of a knot or knotted surface diagram is not changed by the Reidemeister or Roseman moves.
\end{theorem}

Colorings form the foundation for constructing homological invariants.

\begin{definition}(\cite{Nie17a}) \label{homology}
Let $R$ denote a commutative unital ring.
Given a KTQ $(X,[\ ],\eL,\M,\R)$, we can define its homology as follows:\\
Let $C_n(X):=R\langle X^{n+2}\rangle$ be the $R$-module generated freely by $(n+2)$-tuples 
$(x_0,x_1,\ldots, x_n,x_{n+1})$ of elements of $X$. Define
\[
\partial_n^L\xt=\sum_{i=0}^n(-1)^i d_i^{n,L}\xt,
\]
where the formula for $d_i^{n,L}$ is defined inductively by
\begin{align*}
d_0^{n,L}\xt & =(x_1,\ldots,x_{n+1}),\\
d_{i}^{n,L}\xt & =d_{i-1}^{n,L}(x_0,\ldots,x_{i-1},[x_{i-1}x_ix_{i+1}],x_{i+1},\ldots,x_{n+1})
\end{align*}
for $i\in\{1,\ldots,n\}$. We also use the second differential: 
\[
\partial_n^R\xt=\sum_{i=0}^n(-1)^i d_i^{n,R}\xt,
\]
where the formula for $d_i^{n,R}$ is defined inductively by
\begin{align*}
d_n^{n,R}\xt & =(x_0,\ldots,x_{n}),\\
d_{i-1}^{n,R}\xt & =d_{i}^{n,R}(x_0,\ldots,x_{i-1},[x_{i-1}x_ix_{i+1}],x_{i+1},\ldots,x_{n+1})
\end{align*}
for $i\in\{1,\ldots,n\}$.
Next, we combine these differentials in a standard way:
\[
\partial_n\xt=\sum_{i=0}^n(-1)^i (d_i^{n,L}-d_i^{n,R})\xt,
\]
that is we take
\[
\partial_n=\partial_n^L-\partial_n^R.
\]

Let $x$ denote the $(n+2)$-tuple $\xt$. 

We can describe the coordinates of $d_i^{n,L}$ and $d_i^{n,R}$ inductively, which can be useful in concise computer programs.
\[
d_i^{n,L}=(d_{i,1}^{n,L},\ldots,d_{i,k}^{n,L},\ldots,d_{i,n+1}^{n,L})
\] 
is calculated from right to left. For $i\in\{0,\ldots,n\}$ and $k\in\{1,\ldots,n+1\}$,
\begin{equation}\label{leftcoord}
d_{i,k}^{n,L}x = \left\{
\begin{array}{rl}
[x_{k-1}x_k(d_{i,k+1}^{n,L}x)] & \text{if } k\leq  i\\
x_k & \text{if } k > i.
\end{array} \right.
\end{equation}
\[
d_i^{n,R}=(d_{i,0}^{n,R},\ldots,d_{i,k}^{n,R},\ldots,d_{i,n}^{n,R})
\] 
is calculated from left to right. For $i\in\{0,\ldots,n\}$ and $k\in\{0,\ldots,n\}$,
\begin{equation}\label{rightcoord}
d_{i,k}^{n,R}x = \left\{
\begin{array}{rl}
[(d_{i,k-1}^{n,R}x)x_kx_{k+1}] & \text{if } k > i\\
x_k & \text{if } k \leq i.
\end{array} \right.
\end{equation}
\end{definition}

\begin{example} In low dimensions the differential $\partial$ is as follows:
\[
\partial_0(a,b)=b-a,
\]
\begin{align*}
\partial_1(a,b,c)&=(b,c)-(a,[abc])\\
 &-([abc],c)+(a,b),
\end{align*}
\begin{align*} 
\partial_2(a,b,c,d)&=(b,c,d)-(a,[abc],[[abc]cd])\\
&-([abc],c,d)+(a,b,[bcd])\\
&+([ab[bcd]],[bcd],d)-(a,b,c),
\end{align*}
\begin{align*}
\partial_3(a,b,c,d,e)&=(b,c,d,e)-(a,[abc],[[abc]cd],[[[abc]cd]de])\\
&-([abc],c,d,e)+(a,b,[bcd],[[bcd]de])\\
&+([ab[bcd]],[bcd],d,e)-(a,b,c,[cde])\\
&-([ab[bc[cde]]],[bc[cde]],[cde],e)+(a,b,c,d).
\end{align*}
\end{example}

\begin{definition}\cite{Nie17a}\label{degs}
For a ternary quasigroup $(X,[\ ],\eL,\M,\R)$ satisfying axioms LN and RN, and for $n\geq 1$,
let $C_n^D(X,[\ ])$ denote the $R$-module generated freely by $(n+2)$-tuples 
$x=(x_0,x_1,\ldots, x_n,x_{n+1})$ of elements of $X$ satisfying one of the following equivalent conditions:
\begin{itemize}
\item[(D1)] $x$ contains $a$, $b$, $abb\R$ on three consecutive coordinates, for some $a$ and $b\in X$;
\item[(D2)] $x$ contains $a$, $b$, $c$ on three consecutive coordinates, for some $a$, $b$ and $c\in X$, such that $b=[abc]$;
\item[(D3)] $x$ contains $bbc\eL$, $b$, $c$ on three consecutive coordinates, for some $b$ and $c\in X$. 
\end{itemize}
For $n<1$, we take $C_n^D(X,[\ ])=0$.
\end{definition}

The condition (D2) was suggested in \cite{NOO18}, and indeed the equivalence can be seen as follows: $b=[abc]$ if and only if $a=bbc\eL$ and $c=abb\R$. The condition (D2) is convenient to use, and we also note that in this form it can be used to define degenerate part for homology of general algebras satisfying the axioms LN and RN (that is, not necessarily ternary quasigroups; see Theorem \ref{pkdegLR} for a generalization of this fact).

\begin{definition}\label{standardnormalized}
We proved in \cite{Nie17a} that 
$(C_n^D(X,[\ ]),\partial_n)$ is a chain subcomplex of 
$(C_n(X),\partial_n)$, and defined the quotient complex
\[
(C_n^{N}(X,[\ ]),\partial_n)=(C_n(X)/C_n^D(X,[\ ]),\partial_n),
\]
with the induced differential (and the same notation). We called its homology {\it normalized}, and denoted it by $H^N(X,[\ ])$.
\end{definition}

If $(X,[\ ])$ is a KTQ used for a coloring $\mathcal{C}$ of a knot (resp. knotted surface) diagram, then such colored diagram represents a cycle in $H_1^N(X,[\ ])$ (resp.  $H_2^N(X,[\ ])$). The cycle can be written as $\sum_{cr}\tau(cr)cp(cr,\mathcal{C})$ 
(resp. $\sum_{tr}\tau(tr)cp(tr,\mathcal{C})$), where the sum is taken over all crossings (resp. triple points) of the diagram, and $\tau$ denotes the sign of a crossing (resp. triple point). The homology class of such a cycle is not changed by Reidemeister (resp. Roseman) moves, see \cite{Nie17a} for more details.

\section{Layered colorings and the corresponding homology}

\begin{figure}
\begin{center}
\includegraphics[height=4 cm]{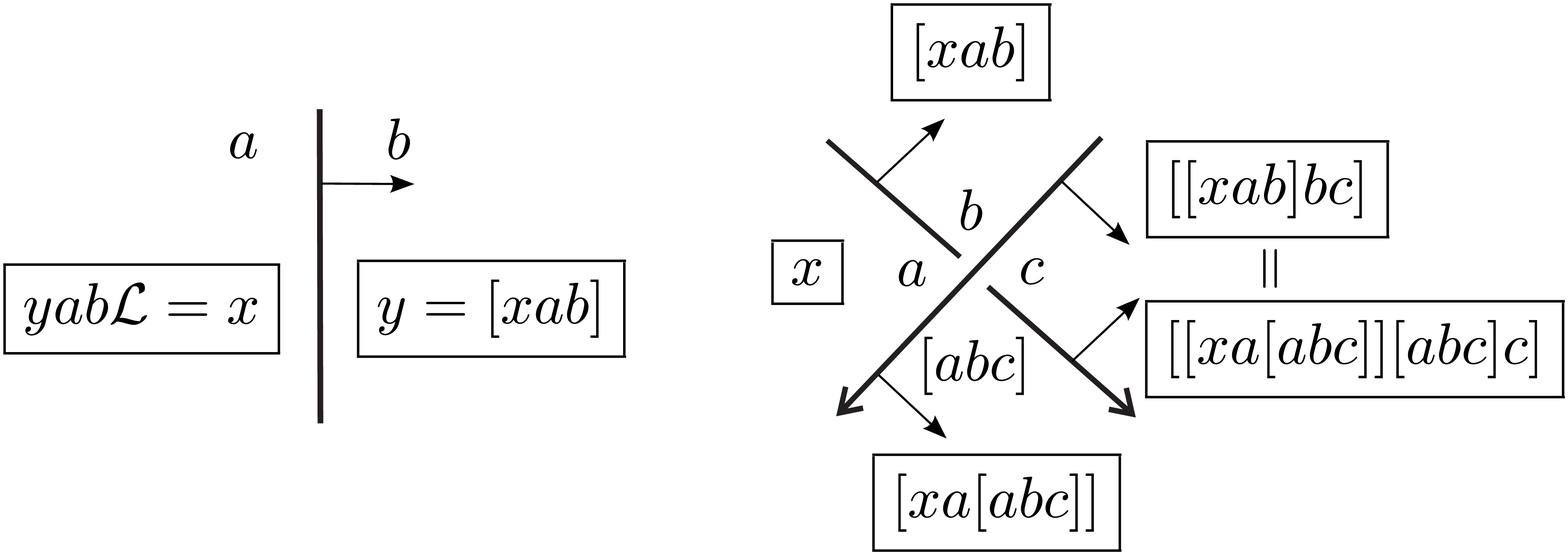}
\caption{The left coloring rule and the condition RN (boxed colors belong to the left coloring).}\label{shadowa}
\end{center}
\end{figure}

\begin{figure}
\begin{center}
\includegraphics[height=4 cm]{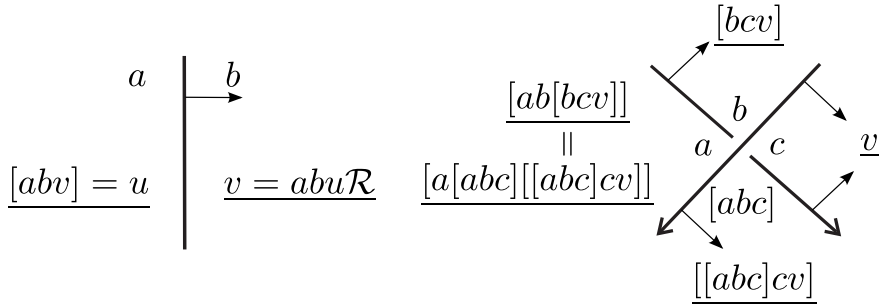}
\caption{The right coloring rule and the condition LN (underlined colors belong to the right coloring).}\label{shadowb}
\end{center}
\end{figure}

\begin{definition} \label{colsys}
Given a knot diagram $D$, a KTQ $(X,[\ ],\eL,\M,\R)$, and a KTQ coloring
$\mathcal{C}\colon Reg(D)\to X$, we can define new KTQ colorings. 

A {\it left coloring}
$\mathcal{C}_L\colon Reg(D)\to X$ is related to $\mathcal{C}$ as on the left side of Fig \ref{shadowa}. More specifically, if the co-orientation arrow points from the region $r_1$ to $r_2$, $\mathcal{C}(r_1)=a$, $\mathcal{C}(r_2)=b$, and $\mathcal{C}_L(r_1)=x$, then $y:=\mathcal{C}_L(r_2)=[xab]$, which is equivalent to $x=yab\eL$. The fact that the KTQ coloring rule is satisfied for the new colors follows from the axiom LN. The right side of Fig. \ref{shadowa}
shows the compatibility of the two ways of extending the coloring $\mathcal{C}_L$ around a crossing thanks to the axiom RN.

A {\it right coloring}
$\mathcal{C}_R\colon Reg(D)\to X$ is obtained from $\mathcal{C}$ as on the left side of Fig. \ref{shadowb}. That is, if the co-orientation arrow points from the region $r_1$ to $r_2$, $\mathcal{C}(r_1)=a$, $\mathcal{C}(r_2)=b$, and $\mathcal{C}_R(r_2)=v$, then $u:=\mathcal{C}_R(r_1)=[abv]$, which is equivalent to $v=abu\R$. The KTQ coloring rule for $\mathcal{C}_R$ is satisfied due to the axiom RN, and the compatibility around a crossing follows from LN.

Similar rules apply to the left and right colorings of a KTQ colored knotted surface diagram $D$ in $\mathbb{R}^3$ (see Fig. \ref{surfleftright}).
\end{definition}

\begin{definition}\label{layered}
Nothing prevents us from considering a left coloring of a left coloring, or a right coloring of a right coloring. Due to our homological conventions, it will be useful to take chains of colorings of $D$ of the form
\[
\mathcal{C}^{(p,k)}:=(\mathcal{C}_0,\mathcal{C}_1,\ldots,\mathcal{C}_{p-1},\mathcal{C}_p,\mathcal{C}_{p+1},\ldots,\mathcal{C}_{p+k+1}),
\]
where $\mathcal{C}_p$ denotes a KTQ coloring that we start with, $\mathcal{C}_{p-1}$ is its left coloring, after which there are inductively defined left colorings
$\mathcal{C}_{p-2}$ (which is a left coloring for $\mathcal{C}_{p-1}$) down to $\mathcal{C}_0$ which is defined last. The right colorings start with $\mathcal{C}_{p+1}$, whose right coloring is $\mathcal{C}_{p+2}$, and end with $\mathcal{C}_{p+k+1}$. That is, $p$ left colorings are followed by a basic KTQ coloring $\mathcal{C}_p$, which is followed by $k$ right colorings. We call such a system of colorings of $D$ a {\it layered coloring}.
\end{definition}

\begin{figure}
\begin{center}
\includegraphics[height=3.7 cm]{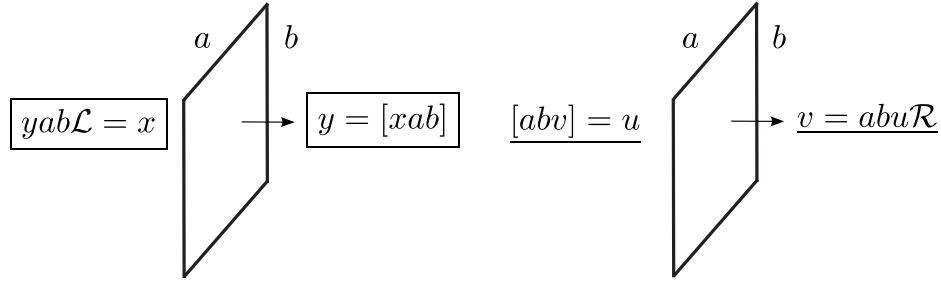}
\caption{The left and right colorings for a surface sheet.}\label{surfleftright}
\end{center}
\end{figure}

\begin{lemma}
Let $D$ denote a classical knot diagram on a plane or on an oriented compact surface, or a knotted surface diagram in $\mathbb{R}^3$. Then the number of layered colorings $\mathcal{C}^{(p,k)}=(\mathcal{C}_0,\mathcal{C}_1,\ldots,\mathcal{C}_{p-1},\mathcal{C}_p,\mathcal{C}_{p+1},\ldots,\mathcal{C}_{p+k+1})$ of $D$ is not changed by the moves applicable to $D$, that is, Reidemeister or Roseman moves.
\end{lemma}
\begin{proof} 
We proved in \cite{Nie17a} that the number of KTQ colorings is an invariant under Reidemeister and Roseman moves. Thus, for a KTQ coloring $\mathcal{C}$ of a diagram $D$, there is a corresponding coloring $\mathcal{C}'$ of a diagram $D'$ after the move. Similarly, if $\mathcal{C}_L$ is a left coloring for $\mathcal{C}$, then it is a KTQ coloring, and there is a corresponding coloring $\mathcal{C}_L'$ for $D'$. The colors of the new regions that may appear after a Reidemeister or Roseman move are uniquely determined by the colors of regions present before the move (see \cite{Nie17a}). Also, a left coloring is uniquely determined by a choice of a color for a single region and the coloring for which it is a left coloring. It follows that $\mathcal{C}_L'$ is a left coloring for $\mathcal{C}'$. Similar considerations apply to the right colorings.  
\end{proof}

Now we will construct a homology, which can be viewed as a generalization of the standard KTQ homology from Def. \ref{standardnormalized}. Layered colorings will yield cycles in this new homology, such that their classes are invariant under Reidemeister and Roseman moves. First, let us recall the definition of a presimplicial module.

\begin{definition}
Let $R$ be a commutative unital ring. A {\it presimplicial module} is a family $C=(C_n)$ of $R$-modules together with face maps $d_i^n\colon C_n\to C_{n-1}$, $i=0,1,\ldots,n$, satisfying 
\[
d_i^{n-1}d_j^n=d_{j-1}^{n-1}d_i^n
\]
for $0\leq i<j\leq n$.
\end{definition}

We proved in \cite{Nie17a} that if $(X,[\ ])$ satisfies the conditions LN and RN, then for $C_n$ and $d_i^n:=d_i^{n,L}-d_i^{n,R}$ as in Def. \ref{homology}, $(C_n,d_i^n)$ is a presimplicial module. We will use the following simple lemma in which we omit some indices $n$ for clarity.

\begin{lemma}
Let $(C_n,d_i)$ be a presimplicial module, and $p, k\geq 0$ be fixed integers. Then
$(C_n,\partial^{(p,k)}_n)$, where $\partial^{(p,k)}_n=\sum_{i=p}^{n-k}(-1)^id_i$, is a chain complex.
\end{lemma}
\begin{proof}
\begin{align*}
\partial^{(p,k)}_{n-1}\partial^{(p,k)}_n &=\sum_{i=p}^{n-1-k}(-1)^id_i(\sum_{j=p}^{n-k}(-1)^jd_j)=\sum_{\substack{p\leq i\leq n-1-k \\ p\leq j\leq n-k}}(-1)^{i+j}d_id_j\\
&=\sum_{p\leq i<j\leq n-k} (-1)^{i+j}d_id_j+ \sum_{p\leq j\leq i\leq n-1-k} (-1)^{i+j}d_id_j\\
&=\sum_{p\leq i<j\leq n-k} (-1)^{i+j}d_{j-1}d_i+ \sum_{p\leq j\leq i\leq n-1-k}(-1)^{i+j}d_id_j\\
&\stackrel{(j'=j-1)}{=}\sum_{p\leq i\leq j'\leq n-1-k} (-1)^{i+j'+1}d_{j'}d_i+ \sum_{p\leq j\leq i\leq n-1-k} (-1)^{i+j}d_id_j=0.
\end{align*}
\end{proof}

\begin{definition}
Let $X$ be a set with a ternary operation $[\ ]\colon X^3\to X$ satisfying the conditions LN and RN.
Define the chain groups $C_n$ and differentials $d_i^{n,L}$, $d_i^{n,R}$ as in Def. \ref{homology}. For $p, k\geq 0$,
we introduce the following notation: $\partial^{(p,k),L}_n=\sum_{i=p}^{n-k}(-1)^id_i^{n,L}$,
$\partial^{(p,k),R}_n=\sum_{i=p}^{n-k}(-1)^id_i^{n,R}$, 
and $\partial^{(p,k)}_n=\sum_{i=p}^{n-k}(-1)^id_i^n$.
We denote the homology of the chain complex $(C_n,\partial^{(p,k)}_n)$ by
$H^{(p,k),u}(X,[\ ])$, where $u$ stands for unnormalized.
\end{definition}

Now we define a weaker form of degeneracy. The idea of late degeneracy in quandle homology was used in \cite{LN03,PrPu16}. In our case, the degeneracy is both late and early at the same time. 

\begin{definition}\label{pkdegs}
For a set $X$ with a ternary operation $[\ ]\colon X^3\to X$, a commutative unital ring $R$, integers $p$, $k\geq 0$, and $n-p-k\geq 1$,
let $C_n^{(p,k),D}(X,[\ ])$ denote the $R$-module generated freely by $(n+2)$-tuples 
$x=(x_0,x_1,\ldots, x_n,x_{n+1})$ of elements of $X$ with an index $i$ such that $p\leq i$, $i+2\leq n+1-k$, and $x_{i+1}=[x_{i}x_{i+1}x_{i+2}]$. In other words, we ignore degeneracy as defined in Def. \ref{degs} if it occurs on the first $p$, or on the last $k$ coordinates of $x$. For $n-p-k<1$, we take $C_n^{(p,k),D}(X,[\ ])=0$.
\end{definition}

\begin{lemma}\label{pkdegL}
Let $X$ be a set with a ternary operation $[\ ]$ satisfying the condition RN. Then
\[
\partial_n^{(p,k),L}(C_n^{(p,k),D}(X,[\ ]))\subset C_{n-1}^{(p,k),D}(X,[\ ]).
\]
\end{lemma}
\begin{proof}
Let $a$, $b$, $c$ be a degeneracy triple (i.e., $b=[abc]$) occuring in the $(n+2)$-tuple $x$ on coordinates with indices $j$, $j+1$, and $j+2$, where $p\leq j$ and $j+2\leq n+1-k$. We need to show that all $d_i^{n,L}x$, for $p\leq i\leq n-k$, either get cancelled, or contain a degeneracy triple that is not on the first $p$ or on the last $k$ coordinates. $d_i^{n,L}x$ contains at the end the sequence
$x_{i+1},\ldots,x_{n+1}$. It follows that $a$, $b$, $c$ occurs also in all
$d_i^{n,L}x$ with $i\in\{0,\ldots, j-1\}$. In all $d_i^{n,L}x$, the first coordinate element of $x$ is removed. Thus, not having to consider $d_i^{n,L}x$ for $i<p$ contributes to the fact that the degeneracy triple in not too early in $d_i^{n,L}x$ in case $j=p$.
Now let $i=j+1$:
\begin{align*}
&d_{j+1}^{n,L}(x_0,\ldots,x_{j-1},a,b,c,x_{j+3},\ldots,x_{n+1})\\
& =d_j^{n,L}(x_0,\ldots,x_{j-1},a,[abc],c,x_{j+3},\ldots,x_{n+1})\\
& =d_j^{n,L}(x_0,\ldots,x_{j-1},a,b,c,x_{j+3},\ldots,x_{n+1}).
\end{align*}
Since in $\partial_n^{(p,k),L}$, $d_{j}^{n,L}$ and $d_{j+1}^{n,L}$ appear with opposite signs, this pair gets removed. Now let $j+2\leq i \leq n-k$. We will show that in $d_{i}^{n,L}x$, the triple $d^{n,L}_{i,j+1}x, d^{n,L}_{i,j+2}x, d^{n,L}_{i,j+3}x$ is a degeneracy triple.
From the formula (\ref{leftcoord}), we get:
\begin{align*}
d_{i,j+1}^{n,L}x & =[x_jx_{j+1}(d_{i,j+2}^{n,L}x)]=
[x_jx_{j+1}[x_{j+1}x_{j+2}(d_{i,j+3}^{n,L}x)]]\\
& =[ab[bc(d_{i,j+3}^{n,L}x)]],
\end{align*}
and 
\[
d_{i,j+2}^{n,L}x =[x_{j+1}x_{j+2}(d_{i,j+3}^{n,L}x)]
=[bc(d_{i,j+3}^{n,L}x)].
\]
Now we check the degeneracy using the condition RN:
\begin{align*}
& [(d^{n,L}_{i,j+1}x) (d^{n,L}_{i,j+2}x) (d^{n,L}_{i,j+3}x)] =[[[ab[bc(d^{n,L}_{i,j+3}x)]][bc(d^{n,L}_{i,j+3}x)](d^{n,L}_{i,j+3}x)]\\
& =[[abc]c(d^{n,L}_{i,j+3}x)]=[bc(d^{n,L}_{i,j+3}x)]=d^{n,L}_{i,j+2}x.
\end{align*}
The degeneracy triple now has increased indices $j+1$, $j+2$, and $j+3$, so it is not too early in $d_i^{n,L}=(d_{i,1}^{n,L},\ldots,d_{i,k}^{n,L},\ldots,d_{i,n+1}^{n,L})$. It is also not too late,
since we have considered $i$ such that $j+3\leq i+1\leq n-k+1$, as needed.
\end{proof}

\begin{lemma}\label{pkdegR}
Let $X$ be a set with a ternary operation $[\ ]$ satisfying the condition LN. Then
\[
\partial_n^{(p,k),R}(C_n^{(p,k),D}(X,[\ ]))\subset C_{n-1}^{(p,k),D}(X,[\ ]).
\]
\end{lemma}
\begin{proof}
The proof is completely symmetric to the proof of Lemma \ref{pkdegL}.
\end{proof}

From Lemmas \ref{pkdegL} and \ref{pkdegR} follows

\begin{theorem}\label{pkdegLR}
Let $X$ be a set with a ternary operation $[\ ]\colon X^3\to X$ satisfying the conditions LN and RN. Then
\[
\partial_n^{(p,k)}(C_n^{(p,k),D}(X,[\ ]))\subset C_{n-1}^{(p,k),D}(X,[\ ]).
\]
\end{theorem}

Now that we have a chain subcomplex depending on $p$ and $k$, we can define a quotient complex.

\begin{definition}
Let $X$ be a set with a ternary operation $[\ ]\colon X^3\to X$ satisfying the conditions LN and RN. 
We define the quotient complex
\[
(C_n^{(p,k)}(X,[\ ]),\partial_n^{(p,k)})=(C_n(X)/C_n^{(p,k),D}(X,[\ ]),\partial_n^{(p,k)}),
\]
with the induced differential (and the same notation). We call its homology {\it truncated}, and denote it by $H^{(p,k)}(X,T)$.
\end{definition}

Let $\mathcal{C}^{(p,k)}=(\mathcal{C}_0,\mathcal{C}_1,\ldots,\mathcal{C}_{p-1},\mathcal{C}_p,\mathcal{C}_{p+1},\ldots,\mathcal{C}_{p+k+1})$ be a layered KTQ coloring of a knot diagram $D$. We will now explain how to assign to it a cycle in the homology group $H^{(p,k)}_n(X,[\ ])$, where $n=p+k+1$. In case $D$ is a knotted surface diagram, the corresponding cycle is from $H^{(p,k)}_n(X,[\ ])$ with $n=p+k+2$.

\begin{definition}
Let $(X,[\ ])$ be a KTQ, and $D$ be an oriented diagram. For a layered coloring
$\mathcal{C}^{(p,k)}=(\mathcal{C}_0,\mathcal{C}_1,\ldots,\mathcal{C}_{p-1},\mathcal{C}_p,\mathcal{C}_{p+1},\ldots,\mathcal{C}_{p+k+1})$ of $D$, and for a given $e$, $cr$, $db$, or $tr$ of $D$, we define its {\it colored paths} $cp(e,\mathcal{C}^{(p,k)})$, $cp(cr,\mathcal{C}^{(p,k)})$, $cp(db,\mathcal{C}^{(p,k)})$, and $cp(tr,\mathcal{C}^{(p,k)})$, by extending the colored paths $cp(e,\mathcal{C}_p)$, $cp(cr,\mathcal{C}_p)$, $cp(db,\mathcal{C}_p)$, and $cp(tr,\mathcal{C}_p)$
of the initial KTQ coloring $\mathcal{C}_p$ in the following way. In each case, the colored path $cp(*,\mathcal{C}_p)$ begins with the color of the source region $\mathcal{C}_p(r_0)$, and ends with the color of the target region $\mathcal{C}_p(r_t)$, where $*$ denotes one of: $e$, $cr$, $db$, $tr$. A $cp(*,\mathcal{C}^{(p,k)})$ begins with the left colorings of the source region, and ends with the right colorings of the target region. That is, we can write $cp(*,\mathcal{C}^{(p,k)})$ as the following sequence:
\[
(\mathcal{C}_0(r_0),\ldots,\mathcal{C}_{p-1}(r_0),\mathcal{C}_{p}(r_0),\mathcal{C}_{p}(r_1),\ldots,\mathcal{C}_{p}(r_t),\mathcal{C}_{p+1}(r_t),\ldots,\mathcal{C}_{p+k+1}(r_t)),
\]
where $r_0,r_1,\ldots,r_t$ are the regions appearing in $cp(*,\mathcal{C}_p)$.
\end{definition}

\begin{figure}
\begin{center}
\includegraphics[height=6.5 cm]{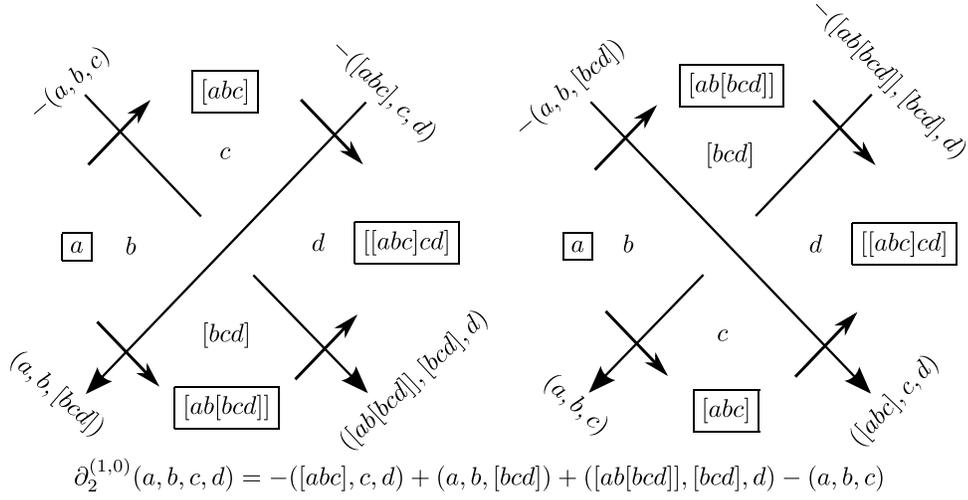}
\caption{The differential of a left-extended colored path of a crossing.}\label{quasidiffleft}
\end{center}
\end{figure}

\begin{figure}
\begin{center}
\includegraphics[height=6.5 cm]{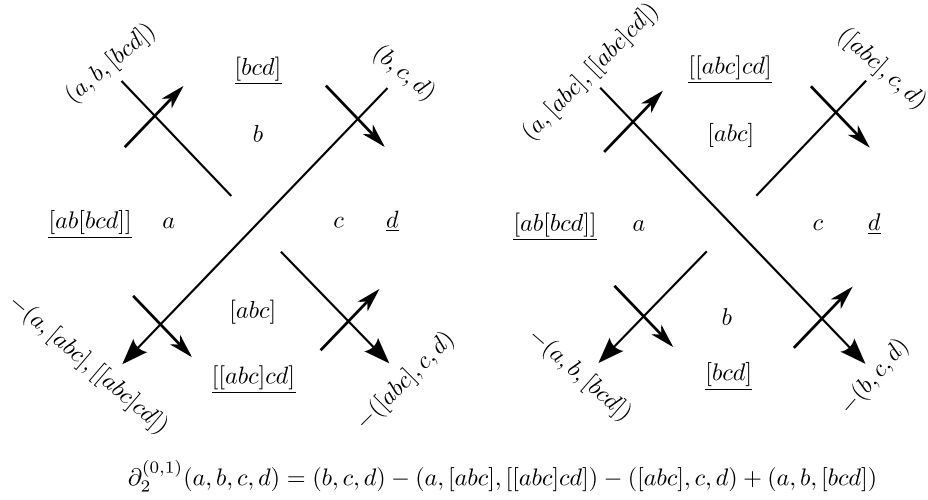}
\caption{The differential of a right-extended colored path of a crossing.}\label{quasidiffright}
\end{center}
\end{figure}

\begin{theorem}\label{knotcycle}
Let $D$ denote an oriented knot diagram on a plane, or on an oriented compact surface $F$. Let $(X,[\ ])$ be a KTQ, and $\mathcal{C}^{(p,k)}$ be a layered coloring of $D$ using elements of $(X,[\ ])$. Then 
\[
c(\mathcal{C}^{(p,k)}):=\sum_{cr}\tau(cr)cp(cr,\mathcal{C}^{(p,k)}),
\]
with the summation taken over all crossings of $D$, is a cycle in 
$H_n^{(p,k)}(X,[\ ])$, where $n=p+k+1$.
\end{theorem}
\begin{proof}
The proof is obtained by considering the differential 
\begin{align*}
\partial_{n=p+k+1}^{(p,k)}(cp(cr,\mathcal{C}^{(p,k)})) & =
(-1)^p d_p^{n,L}(cp(cr,\mathcal{C}^{(p,k)}))-(-1)^p d_p^{n,R}(cp(cr,\mathcal{C}^{(p,k)}))\\
& +(-1)^{p+1}d_{p+1}^{n,L}(cp(cr,\mathcal{C}^{(p,k)}))-(-1)^{p+1} d_{p+1}^{n,R}(cp(cr,\mathcal{C}^{(p,k)}))
\end{align*}
of a colored path of a positive crossing. It is a sum of signed colored paths of edges of the crossing. If $e_{inL}$, $e_{inR}$, $e_{outL}$, $e_{outR}$ denote, respectively,
the left incoming edge, the right incoming edge, the left outgoing edge, and the right outgoing edge of the crossing, as in the Figures \ref{quasidiffleft} and 
\ref{quasidiffright}, then
\begin{align*}
& d_p^{n,L}(cp(cr,\mathcal{C}^{(p,k)}))=cp(e_{inR},\mathcal{C}^{(p,k)}),\\
& d_p^{n,R}(cp(cr,\mathcal{C}^{(p,k)}))=cp(e_{outL},\mathcal{C}^{(p,k)}),\\
& d_{p+1}^{n,L}(cp(cr,\mathcal{C}^{(p,k)}))=cp(e_{outR},\mathcal{C}^{(p,k)}),\\
& d_{p+1}^{n,R}(cp(cr,\mathcal{C}^{(p,k)}))=cp(e_{inL},\mathcal{C}^{(p,k)}).
\end{align*}
Thus, the colored paths of the incoming edges have opposite signs to the colored paths of the outgoing edges in the differential, which leads to a reduction when an incoming (resp. outgoing) edge becomes outgoing (resp. incoming) at its next crossing, and their colored paths remain the same.
The same holds for a negative crossing. 
\end{proof}

\begin{figure}
\begin{center}
\includegraphics[height=4.5 cm]{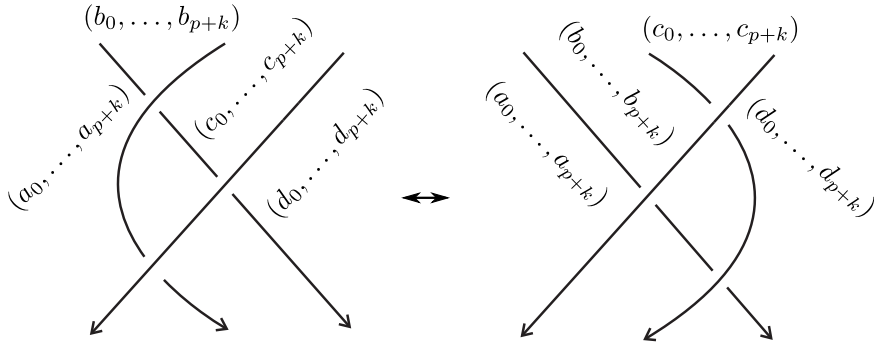}
\caption{Labels for $\mathcal{C}^{(p,k)}$ in the third Reidemeister move}\label{rad3brc}
\end{center}
\end{figure}

\begin{theorem}
The homology class in $H_{p+k+1}^{(p,k)}(X,[\ ])$ of the cycle $c(\mathcal{C}^{(p,k)})$ as defined in Thm \ref{knotcycle} is not changed by the Reidemeister moves.
\end{theorem}

\begin{proof}
The first Reidemeister move adds or removes a degenerate cycle that consists of a single $(p+k+3)$-tuple. The second Reidemeister move involves two crossings with opposite signs but equal colored paths. Now consider the third Reidemeister move in which all the crossings are positive, with the coloring labels as in Fig. \ref{rad3brc}. The contributions from the crossings after the move, minus the contributions before the move are equal (up to the sign) to the boundary 
\[
\partial_{p+k+2}^{(p,k)}(a_0,\ldots,a_{p-1},a_p,b_p,c_p,d_p,d_{p+1},\ldots,d_{p+k}).
\]
\end{proof}

\begin{theorem}\label{surfacecycle}
Let $D$ denote an oriented knotted surface diagram in $\mathbb{R}^3$. Let $(X,[\ ])$ be a KTQ, and $\mathcal{C}^{(p,k)}$ be a layered coloring of $D$ using elements of $(X,[\ ])$. Then 
\[
c(\mathcal{C}^{(p,k)}):=\sum_{tr}\tau(tr)cp(tr,\mathcal{C}^{(p,k)}),
\]
with the summation taken over all triple points of $D$, is a cycle in 
$H_n^{(p,k)}(X,[\ ])$, where $n=p+k+2$.
\end{theorem}

\begin{proof}
The proof is obtained by considering the differential 
\[
\partial_{p+k+2}^{(p,k)}(cp(tr,\mathcal{C}^{(p,k)}))
\] 
of a colored path of a triple point, which is a sum of suitably signed colored paths of double point edges near this triple point. In a neighborhood of a triple point, there are six double point curves: three incoming and three outgoing. In the differential, the signs of the colored paths of the incoming double point edges are opposite to the signs of the colored paths of the outgoing edges. It follows that if a double point edge ends with triple points, its colored path appears twice in $\partial_{p+k+2}^{(p,k)}(c(\mathcal{C}^{(p,k)}))$, but with opposite signs. If a double point edge ends in a branch point, then its colored path is a degenerate cycle.
\end{proof}

\begin{theorem}
The homology class in $H_{p+k+2}^{(p,k)}(X,[\ ])$ of the cycle $c(\mathcal{C}^{(p,k)})$ as defined in Thm \ref{surfacecycle} is not changed by the Roseman moves.
\end{theorem}

\begin{proof}
The proof is similar to the one for standard homology in \cite{Nie17a}, but it uses higher differentials and the boundary of a sequence of colors extending the one from \cite{Nie17a} by adjoining the left colors of its initial region at the beginning, and the right colors of its final region at the end.
\end{proof}

Cohomology groups for the theory described above are defined in a usual dual way, and lead to the cocycle invariants (see \cite{Nie17a}, where it was done for the basic KTQ (co)homology, see also \cite{CJKLS03} for the case of quandle (co)homology).

\section{Computational examples}
In this section we present some calculations using GAP \cite{GAP4}.
Our main tool is a homological classification of cycles associated with colorings of knot diagrams. Suppose that a knot diagram $D$ has $m$ layered colorings $\mathcal{C}^{p,k}$, and thus also $m$ associated cycles in $H_{p+k+1}^{(p,k)}(X,[\ ])$. Then of interest is the partition of $m$ into positive integers $\{m_1,m_2,\ldots,m_l\}$ describing the fact that there are $m_i$ coloring cycles, for 
$1\leq i \leq l$, that are homologically the same.

The first KTQ $(X,[\ ])$ that we will use has five elements, and its multiplication cube can be sliced into the following five matrices, each matrix for a fixed first coordinate of 
$[xyz]$.
For example $[123]=1$ and $[234]=3$.
\[ 
\begin{array}{|c| c ccccc} 
[1yz]&1 & 2 & 3 & 4 & 5 \\
\hline 
1 & 1 & 3 & 4 & 5 & 2 \\
2 & 2 & 4 & 1 & 3 & 5 \\
3 & 3 & 2 & 5 & 1 & 4 \\
4 & 4 & 5 & 3 & 2 & 1 \\ 
5 & 5 & 1 & 2 & 4 & 3 \\ 
\end{array}
\begin{array}{|c| c ccccc} 
[2yz]&1 & 2 & 3 & 4 & 5 \\
\hline 
1 & 3 & 2 & 5 & 1 & 4 \\
2 & 4 & 5 & 3 & 2 & 1 \\
3 & 2 & 4 & 1 & 3 & 5 \\
4 & 5 & 1 & 2 & 4 & 3 \\ 
5 & 1 & 3 & 4 & 5 & 2 \\ 
\end{array}
\begin{array}{|c| c ccccc} 
[3yz]&1 & 2 & 3 & 4 & 5 \\
\hline 
1 & 4 & 5 & 3 & 2 & 1 \\
2 & 1 & 3 & 4 & 5 & 2 \\
3 & 5 & 1 & 2 & 4 & 3 \\
4 & 3 & 2 & 5 & 1 & 4 \\ 
5 & 2 & 4 & 1 & 3 & 5 \\ 
\end{array}
\]

\[
\begin{array}{|c| c ccccc} 
[4yz]&1 & 2 & 3 & 4 & 5 \\
\hline 
1 & 5 & 1 & 2 & 4 & 3 \\
2 & 3 & 2 & 5 & 1 & 4 \\
3 & 1 & 3 & 4 & 5 & 2 \\
4 & 2 & 4 & 1 & 3 & 5 \\ 
5 & 4 & 5 & 3 & 2 & 1 \\ 
\end{array}
\begin{array}{|c| c ccccc} 
[5yz]&1 & 2 & 3 & 4 & 5 \\
\hline 
1 & 2 & 4 & 1 & 3 & 5 \\
2 & 5 & 1 & 2 & 4 & 3 \\
3 & 4 & 5 & 3 & 2 & 1 \\
4 & 1 & 3 & 4 & 5 & 2 \\ 
5 & 3 & 2 & 5 & 1 & 4 \\ 
\end{array}
\]

Its standard first homology group $H_1^N(X,[\ ])=H_1^{(0,0)}(X,[\ ])$ has no torsion part. On the other hand, the torsion part in both $H_2^{(1,0)}(X,[\ ])$ and $H_2^{(0,1)}(X,[\ ])$ is $\mathbb{Z}_5$.

We analyzed colorings of links with two components that have up to eight crossings in the table from \cite{Git04} (there are 48 of them). All are recognized as nontrivial by this KTQ, either by the number of the standard KTQ colorings, or by using $H_2^{(1,0)}(X,[\ ])$ (in this case $H_2^{(0,1)}(X,[\ ])$ works equally well). The number of the standard KTQ colorings $\mathcal{C}^{(0,0)}$ for these links is either 25 or 125. Having 25 colorings proves nontriviality, as the number of colorings for a trivial link with two components is 125. For the ones with 125 colorings, the classification of the corresponding cycles in $H_1^{(0,0)}(X,[\ ])$ does not help, as they are all homologically trivial. In such cases, however, there are 625 colorings $\mathcal{C}^{(1,0)}$, and the partition for the corresponding coloring cycles is $\{225,200,200\}$, which shows that the links are not trivial.

If a knot diagram is on the plane, then instead of considering all colorings, we can just look at the ones with a fixed color of the outside region. For example, knots $7_4$ and $8_8$ both have 25 homologically trivial colorings $\mathcal{C}^{(0,0)}$
with the color of the outside region equal to 1. But when we look at 125 colorings 
$\mathcal{C}^{(1,0)}$, with the color of the outside region equal to 1 for the primary colorings, and arbitrary for the left colorings, then the partitions for the corresponding cycles in $H_2^{(1,0)}(X,[\ ])$ are $\{45,40,40\}$ and $\{125\}$, respectively.

Our second example uses a six-element KTQ $(X,[\ ])$ described by the following tables.

\[ 
\begin{array}{|c| c cccccc} 
[1yz]&1 & 2 & 3 & 4 & 5 & 6 \\
\hline 
 1 & 3 & 1 & 6 & 5 & 4 & 2 \\
 2 & 4 & 2 & 5 & 6 & 3 & 1 \\
 3 & 1 & 3 & 2 & 4 & 5 & 6 \\
 4 & 2 & 4 & 1 & 3 & 6 & 5 \\
 5 & 6 & 5 & 3 & 1 & 2 & 4 \\
 6 & 5 & 6 & 4 & 2 & 1 & 3 \\
\end{array}
\begin{array}{|c| c cccccc} 
[2yz]&1 & 2 & 3 & 4 & 5 & 6 \\
\hline 
 1 & 4 & 5 & 2 & 1 & 3 & 6 \\
 2 & 3 & 6 & 1 & 2 & 4 & 5 \\
 3 & 5 & 4 & 6 & 3 & 1 & 2 \\
 4 & 6 & 3 & 5 & 4 & 2 & 1 \\
 5 & 2 & 1 & 4 & 5 & 6 & 3 \\
 6 & 1 & 2 & 3 & 6 & 5 & 4 \\
\end{array}
\begin{array}{|c| c cccccc} 
[3yz]&1 & 2 & 3 & 4 & 5 & 6 \\
\hline 
 1 & 6 & 2 & 3 & 4 & 5 & 1 \\
 2 & 5 & 1 & 4 & 3 & 6 & 2 \\
 3 & 2 & 6 & 1 & 5 & 4 & 3 \\
 4 & 1 & 5 & 2 & 6 & 3 & 4 \\
 5 & 3 & 4 & 6 & 2 & 1 & 5 \\
 6 & 4 & 3 & 5 & 1 & 2 & 6 \\
\end{array}
\]

\[
\begin{array}{|c| c cccccc} 
[4yz]&1 & 2 & 3 & 4 & 5 & 6 \\
\hline 
 1 & 5 & 4 & 1 & 2 & 6 & 3 \\
 2 & 6 & 3 & 2 & 1 & 5 & 4 \\
 3 & 4 & 5 & 3 & 6 & 2 & 1 \\
 4 & 3 & 6 & 4 & 5 & 1 & 2 \\
 5 & 1 & 2 & 5 & 4 & 3 & 6 \\
 6 & 2 & 1 & 6 & 3 & 4 & 5 \\ 
\end{array}
\begin{array}{|c| c cccccc} 
[5yz]&1 & 2 & 3 & 4 & 5 & 6 \\
\hline 
 1 & 1 & 3 & 5 & 6 & 2 & 4 \\
 2 & 2 & 4 & 6 & 5 & 1 & 3 \\
 3 & 3 & 1 & 4 & 2 & 6 & 5 \\
 4 & 4 & 2 & 3 & 1 & 5 & 6 \\
 5 & 5 & 6 & 1 & 3 & 4 & 2 \\
 6 & 6 & 5 & 2 & 4 & 3 & 1 \\ 
\end{array}
\begin{array}{|c| c cccccc} 
[6yz]&1 & 2 & 3 & 4 & 5 & 6 \\
\hline 
 1 & 2 & 6 & 4 & 3 & 1 & 5 \\
 2 & 1 & 5 & 3 & 4 & 2 & 6 \\
 3 & 6 & 2 & 5 & 1 & 3 & 4 \\
 4 & 5 & 1 & 6 & 2 & 4 & 3 \\
 5 & 4 & 3 & 2 & 6 & 5 & 1 \\
 6 & 3 & 4 & 1 & 5 & 6 & 2 \\
\end{array}
\]

In this case, the torsion part of $H_1^N(X,[\ ])=H_1^{(0,0)}(X,[\ ])$
is $\mathbb{Z}_3$; for $H_2^{(1,0)}(X,[\ ])$ it is $\mathbb{Z}_3^2$, and for
$H_2^{(0,1)}(X,[\ ])$ it is $\mathbb{Z}_3^2+\mathbb{Z}_9^2$.

To save the computation time in case of knot diagrams on the plane, we can use long knots instead, and fix the coloring path $(\mathcal{C}(r_0),\mathcal{C}(r_1))$ of the first arc of the long knot. Then we can also fix the colors of the (unbounded) region $r_0$ for the subsequent left and right colorings. Consider the knots $3_1$, $6_1$, and $7_4$, with braid words $\sigma_{1}\sigma_{1}\sigma_{1}$,
$\sigma_{1}\sigma_{1}\sigma_{2}\sigma_{-1}\sigma_{-3}\sigma_{2}\sigma_{-3}$, and
$\sigma_{1}\sigma_{1}\sigma_{2}\sigma_{-1}
\sigma_{2}\sigma_{2}\sigma_{3}\sigma_{-2}\sigma_{3}$, respectively. With the vertical positioning of the braid, and the top to bottom orientation, the first strand serves as the axis for the long knot, and we can take $(\mathcal{C}(r_0),\mathcal{C}(r_1))$ to be the colors of the top two left regions of the braid. In our computations we took $(\mathcal{C}(r_0),\mathcal{C}(r_1))=(1,2)$, and fixed the color of the region $r_0$ for both left and right colorings: $\mathcal{C}_L(r_0)=\mathcal{C}_R(r_0)=2$. With these restrictions, the three knots all have three colorings $\mathcal{C}^{(0,0)}$ (not distinguished by $H_1^{(0,0)}(X,[\ ])$), and three colorings $\mathcal{C}^{(1,0)}$, with respective partitions for the corresponding cycles in $H_2^{(1,0)}(X,[\ ])$ as follows: $\{1,2\}$, $\{3\}$, $\{1,2\}$.
$H_2^{(0,1)}(X,[\ ])$ proved to be stronger in this case, as the partitions for the cycles assigned to colorings $\mathcal{C}^{(0,1)}$ are, respectively:
$\{1,1,1\}$, $\{3\}$, $\{1,2\}$.

\bibliography{layers}

\begin{thebibliography}{10}

\bibitem{Bel}
V.~D. Belousov.
\newblock {\em {$n$}-arnye kvazigruppy.}
\newblock Izdat. ``Stiinca'', Kishinev, 1972.

\bibitem{BelSan}
V.~D. Belousov and M.~D. Sandik.
\newblock {$N$}-ary quasi-groups and loops.
\newblock {\em Sibirsk. Mat. \v Z.}, 7:31--54, 1966.

\bibitem{CJKLS03}
J.~Scott Carter, Daniel Jelsovsky, Seiichi Kamada, Laurel Langford, and
  Masahico Saito.
\newblock Quandle cohomology and state-sum invariants of knotted curves and
  surfaces.
\newblock {\em Trans. Amer. Math. Soc.}, 355(10):3947--3989, 2003.

\bibitem{CKS01}
J.~Scott Carter, Seiichi Kamada, and Masahico Saito.
\newblock Geometric interpretations of quandle homology.
\newblock {\em J. Knot Theory Ramifications}, 10(3):345--386, 2001.

\bibitem{KnSurf}
J.~Scott Carter and Masahico Saito.
\newblock {\em Knotted surfaces and their diagrams}, volume~55 of {\em
  Mathematical Surveys and Monographs}.
\newblock American Mathematical Society, Providence, RI, 1998.

\bibitem{Dev09}
Rostislav Deviatov.
\newblock Combinatorial knot invariants that detect trefoils.
\newblock {\em J. Knot Theory Ramifications}, 18(9):1193--1203, 2009.

\bibitem{FRS95}
Roger Fenn, Colin Rourke, and Brian Sanderson.
\newblock Trunks and classifying spaces.
\newblock {\em Appl. Categ. Structures}, 3(4):321--356, 1995.

\bibitem{GAP4}
The GAP~Group.
\newblock {\em {GAP -- Groups, Algorithms, and Programming, Version 4.10.1}},
  2019.

\bibitem{Git04}
Thomas~A. Gittings.
\newblock Minimum braids: A complete invariant of knots and links.
\newblock arXiv:math/0401051 [math.GT].

\bibitem{Joy82}
David Joyce.
\newblock A classifying invariant of knots, the knot quandle.
\newblock {\em J. Pure Appl. Algebra}, 23(1):37--65, 1982.

\bibitem{Kam02}
Seiichi Kamada.
\newblock Knot invariants derived from quandles and racks.
\newblock In {\em Invariants of knots and 3-manifolds ({K}yoto, 2001)},
  volume~4 of {\em Geom. Topol. Monogr.}, pages 103--117. Geom. Topol. Publ.,
  Coventry, 2002.

\bibitem{KLT15}
Seiichi Kamada, Victoria Lebed, and Kokoro Tanaka.
\newblock The shadow nature of positive and twisted quandle invariants of
  knots.
\newblock {\em J. Knot Theory Ramifications}, 24(10):1540001, 15, 2015.

\bibitem{LN03}
R.~A. Litherland and Sam Nelson.
\newblock The {B}etti numbers of some finite racks.
\newblock {\em Journal of Pure and Applied Algebra}, 178(2):187 -- 202, 2003.

\bibitem{Mat82}
S.~V. Matveev.
\newblock Distributive groupoids in knot theory.
\newblock {\em Mat. Sb. (N.S.)}, 119(161)(1):78--88, 160, 1982.

\bibitem{NeNe17}
Deanna Needell and Sam Nelson.
\newblock Biquasiles and dual graph diagrams.
\newblock {\em J. Knot Theory Ramifications}, 26(08):1750048, 2017.

\bibitem{NOO18}
Sam Nelson, Kanako Oshiro, and Natsumi Oyamaguchi.
\newblock Local biquandles and {N}iebrzydowski's tribracket theory.
\newblock arXiv:1809.09442v1 [math.GT].

\bibitem{Nie17a}
Maciej Niebrzydowski.
\newblock Homology of ternary algebras yielding invariants of knots and knotted
  surfaces.
\newblock arXiv:1706.04307 [math.GT], submitted in July 2017.

\bibitem{Nie14}
Maciej Niebrzydowski.
\newblock On some ternary operations in knot theory.
\newblock {\em Fund. Math.}, 225(1):259--276, 2014.

\bibitem{PrPu16}
J\'{o}zef~H. Przytycki and Krzysztof~K. Putyra.
\newblock The degenerate distributive complex is degenerate.
\newblock {\em Eur. J. Math.}, 2(4):993--1012, 2016.

\bibitem{RS00}
Colin Rourke and Brian Sanderson.
\newblock A new classification of links and some calculations using it.
\newblock arXiv:math/0006062 [math.GT].

\bibitem{Sm08}
Jonathan D.~H. Smith.
\newblock Ternary quasigroups and the modular group.
\newblock {\em Comment. Math. Univ. Carolin.}, 49(2):309--317, 2008.

\end{thebibliography}
\bibliographystyle{plain}
\end{document}